\documentclass[a4paper,english]{extarticle}
\usepackage[latin9]{inputenc}
\usepackage{babel}
\usepackage{mathrsfs}
\usepackage{amsthm}
\usepackage{amsmath}
\usepackage{amssymb}
\usepackage{wasysym}
\usepackage{esint}
\usepackage[unicode=true,pdfusetitle,
 bookmarks=true,bookmarksnumbered=false,bookmarksopen=false,
 breaklinks=false,pdfborder={0 0 1},backref=false,colorlinks=false]
 {hyperref}

\makeatletter

\pdfpageheight\paperheight
\pdfpagewidth\paperwidth

\numberwithin{equation}{section}
\numberwithin{figure}{section}
\theoremstyle{plain}
\newtheorem{thm}{\protect\theoremname}[section]
\theoremstyle{plain}
\newtheorem{lem}[thm]{\protect\lemmaname}
\theoremstyle{definition}
\newtheorem{defn}[thm]{\protect\definitionname}
\theoremstyle{plain}
\newtheorem{prop}[thm]{\protect\propositionname}
\ifx\proof\undefined
\newenvironment{proof}[1][\protect\proofname]{\par
\normalfont\topsep6\p@\@plus6\p@\relax
\trivlist
\itemindent\parindent
\item[\hskip\labelsep\scshape #1]\ignorespaces
}{%
\endtrivlist\@endpefalse
}
\providecommand{\proofname}{Proof}
\fi

\usepackage{cancel}
\numberwithin{equation}{section}
\date{}

\makeatother

\providecommand{\definitionname}{Definition}
\providecommand{\lemmaname}{Lemma}
\providecommand{\propositionname}{Proposition}
\providecommand{\theoremname}{Theorem}

\begin{document}

\title{Nearly  cloaking for the elasticity system with residual stress}

\author{Yi-Hsuan Lin}
\maketitle
\begin{abstract}
The nearly cloaking via transformation optics approach for the isotropic
elastic wave fields is considered. This work extends the study of
the nearly cloaking scheme to the elasticity system with residual
stress in $\mathbb{R}^{N}$ for $N=2,3$, which is an anisotropic
elasticity system. It is worth to mention that there are no minor
symmetric properties for the elastic tensor with residual stress and
this system is invariant under a coordinate transformation. Therefore,
we think the elasticity system residual stress model is more natural
than the isotropic elasticity system in designing the elastic cloaking
medium in the physical sense. In addition, the main difficulty of
treating this problem lies in the fact that there are no layer potential
theory for the residual stress system. Instead, we will derive suitable
elliptic estimates for the elasticity system residual stress by comparing
with the Lamé system to achieve desired results.  
\end{abstract}
\textbf{Key words}: Elastic cloaking, nearly cloaking, anisotropic
elasticity system, residual stress\\
\textbf{Mathematics Subject Classification}: 74B10, 35R30, 35J25,
35R30, 74J20.

\section{Introduction}

This work is concerned with the cloaking theory of the elastic waves
with residual stress in $\mathbb{R}^{N}$ for $N=2,3$. We formulate
the mathematical problem in the following. Let $\Omega$ a bounded
simply connected $C^{\infty}$-smooth domain in $\mathbb{R}^{N}$,
for $N=2,3$ and $u(x)=(u_{i}(x))_{i=1}^{N}$ is the displacement
vector field. Consider the boundary value time-harmonic elasticity
system 
\begin{equation}
\begin{cases}
\sum_{j,k,l=1}^{N}\dfrac{\partial}{\partial x_{j}}\left(C_{ijkl}(x)\dfrac{\partial u_{k}}{\partial x_{l}}\right)+\kappa^{2}\rho u_{i}=0 & \mbox{ in }\Omega\mbox{ for }i=1,2,\cdots,N,\\
\mathcal{N}_{\mathcal{C}}u=\phi\in H^{-1/2}(\partial\Omega)^{N} & \mbox{ on }\partial\Omega,
\end{cases}\label{eq:Main Equations}
\end{equation}
where $\mathcal{C}=(C_{ijkl})$ is a four tensor, $\kappa\in\mathbb{R}$
is the frequency, $\rho=\rho(x)$ denotes the density of the medium
and $\mathcal{N}_{\mathcal{C}}u$ is the Neumann data defined as 
\[
\mathcal{N}_{\mathcal{C}}u:=\left(\sum_{j,k,l=1}^{N}\nu_{j}C_{1jkl}\dfrac{\partial u_{k}}{\partial x_{l}},\sum_{j,k,l=1}^{N}\nu_{j}C_{2jkl}\dfrac{\partial u_{k}}{\partial x_{l}},\cdots,\sum_{j,k,l=1}^{N}\nu_{j}C_{Njkl}\dfrac{\partial u_{k}}{\partial x_{l}}\right)
\]
 is the \textit{boundary traction} on $\partial\Omega$ with $\nu=(\nu_{1},\nu_{2},\cdots,\nu_{N})$
denoting the unit outer normal on $\partial\Omega$. In (\ref{eq:Main Equations}),
$\rho=\rho_{R}+\sqrt{-1}\rho_{I}$ is a complex-valued function with
$\rho_{R}>0$ and $\rho_{I}\geq0$. In order to simplify notations,
we denote $\nabla\cdot(\mathcal{C}\nabla u)$ componentwisely by 
\[
\nabla\cdot(\mathcal{C}\nabla u)_{i}=\sum_{j,k,l=1}^{N}\dfrac{\partial}{\partial x_{j}}\left(C_{ijkl}(x)\dfrac{\partial u_{k}}{\partial x_{l}}\right)\mbox{ for }i=1,2,\cdots,N.
\]
In addition, for any $v=(v_{1},v_{2},\cdots,v_{N})\in H^{1}(\Omega)^{N}$,
we can rewrite (\ref{eq:Main Equations}) by the variational formula
\begin{equation}
\mathcal{B}_{\mathcal{C}}(u,v):=\int_{\Omega}\left\{ \sum_{i.j.k.l=1}^{N}C_{ijkl}\dfrac{\partial u_{k}}{\partial x_{l}}\dfrac{\partial\overline{v_{i}}}{\partial x_{j}}-\kappa^{2}\rho u_{i}\overline{v_{i}}\right\} dx=\int_{\partial\Omega}\phi\cdot\overline{v}dS,\label{eq:Variational Equation}
\end{equation}
where $\overline{v_{i}}$ is the complex conjugate of $v_{i}$ for
$i=1,2,\cdots,N$. Recall that the strong convexity condition is given
as: there exists $c_{0}>0$ such that for all $N\times N$ symmetric
matrix $\varepsilon=(\varepsilon_{ij})_{i,j=1}^{N}$, 
\begin{equation}
\sum_{i,j,k,l=1}^{N}C_{ijkl}\varepsilon_{ij}\varepsilon_{kl}\geq c_{0}\sum_{i,j=1}^{N}|\varepsilon_{ij}|^{2},\mbox{ for all }x\in\Omega.\label{eq:Strong convexity}
\end{equation}
Via the strong convexity condition (\ref{eq:Strong convexity}) and
the G$\mathring{\mathrm{a}}$rding's inequality 
\[
\mathcal{B}_{\mathcal{C}}(u,u)\geq c_{0}\sum_{i,j=1}^{N}\|\nabla u\|_{L^{2}(\Omega)^{N\times N}}^{2}-\kappa^{2}\|\rho\|_{L^{\infty}(\Omega)}\|u\|_{L^{2}(\Omega)^{N}}^{2}\mbox{ for all }u\in H^{1}(\Omega)^{N},
\]
then there exists a unique weak solution to (\ref{eq:Variational Equation})
for all frequency $\kappa\in\mathbb{R}_{+}$ except for a discrete
set $\mathcal{D}$ with the accumulating point at infinity. By using
the well-posed property of (\ref{eq:Main Equations}), we can define
the boundary Neumann-to-Dirichlet (NtD) map as 
\begin{equation}
\Lambda_{\mathcal{C},\rho}:H^{-1/2}(\partial\Omega)^{N}\to H^{1/2}(\partial\Omega)^{N}\mbox{ with }\Lambda_{\mathcal{C},\rho}\phi=u|_{\partial\Omega},\label{eq:NtD map}
\end{equation}
where $u\in H^{1}(\Omega)^{N}$ is the unique solution of (\ref{eq:Main Equations}).

In this paper, for the elasticity system with residual stress, we
consider the following boundary value problem (\ref{eq:Main Equations})
with the elasticity tensor $\mathcal{C}(x)=(C_{ijkl}(x))_{i,j,k,l=1}^{N}$
and 
\begin{equation}
C_{ijkl}(x)=\lambda(x)\delta_{ij}\delta_{kl}+\mu(x)(\delta_{ik}\delta_{jl}+\delta_{il}\delta_{jk})+t_{jl}(x)\delta_{ik},\label{eq:Residual stress four tensor}
\end{equation}
where $\delta_{ij}$ is the Kronecker delta and $\lambda(x),\mu(x)$
are the Lamé parameters and $\mathcal{C}(x)$ satisfies the strong
convexity condition (\ref{eq:Strong convexity}). The second-rank
tensor $T(x)=(t_{jl}(x))_{j,l=1}^{N}$ is the residual stress and
satisfies the following conditions: 
\begin{enumerate}
\item Symmetry: 
\begin{equation}
t_{jl}(x)=t_{lj}(x),\mbox{ for }j,l=1,2,\cdots,N,\mbox{ for all }x\in\Omega.\label{eq:Symmetric}
\end{equation}

\item Divergence free:
\begin{equation}
\nabla\cdot T(x)=\sum_{l}\partial_{x_{l}}t_{jl}(x)=0,\mbox{ for }j=1,2,\cdots,N,\mbox{ for all }x\in\Omega.\label{eq:Divergence free}
\end{equation}

\item Vanishing boundary: 
\begin{equation}
T(x)\cdot\nu=0\mbox{ or }\sum_{l=1}^{N}t_{jl}(x)\nu_{l}=0\mbox{ for all }x\in\partial\Omega\label{eq:zero traction boundary}
\end{equation}
for $j=1,2,\cdots,N$, where $\nu=(\nu_{1},\nu_{2},\cdots,\nu_{N})$
is a unit outer normal on $\partial\Omega$. In fact, (\ref{eq:Divergence free})
and (\ref{eq:zero traction boundary}) can be expressed weakly by
\[
\int_{\Omega}T\cdot\nabla vdx=0
\]
for all $v\in H^{1}(\Omega)^{N}$.
\end{enumerate}
Moreover, for $N=2,3$, we also assume that the Lamé moduli satisfy
the strong convexity condition 
\begin{equation}
\mu(x)\geq c_{0}>0\mbox{ and }N\lambda(x)+2\mu(x)\geq c_{0}>0\mbox{ for all }x\in\Omega.\label{eq:strong convexity for Lame}
\end{equation}
For more details about the strong convexity property, we refer readers
to \cite{tanuma2007stroh}. It is easy to see when (\ref{eq:Residual stress four tensor})
is the elastic four tensor, (\ref{eq:Main Equations}) is an \textit{anisotropic}
elasticity system. By the symmetric property (\ref{eq:Symmetric}),
for the elastic tensor with residual stress (\ref{eq:Residual stress four tensor}),
we have the major symmetric property without minor symmetry, which
means 
\begin{equation}
C_{ijkl}=C_{klij}\mbox{ but }C_{ijkl}\neq C_{jikl}\mbox{ for any }i,j,k,l=1,2,\cdots,N.\label{eq:Symmetry for residual stress}
\end{equation}
In the homogeneous medium, $\lambda(x)\equiv\lambda^{(0)},\mu(x)\equiv\mu^{(0)}$
are the Lamé constants and $T=T(x)=(t_{jl}(x))\in W^{2,\infty}(\Omega)^{N\times N}$
is an arbitrary residual stress coefficient satisfying (\ref{eq:Symmetric})-(\ref{eq:zero traction boundary})
for $i,j=1,2,\cdots,N$. Let $\mathcal{C}(x)\equiv\mathcal{C}^{(0)}$
with 
\begin{equation}
C_{ijkl}^{(0)}=\lambda^{(0)}\delta_{ij}\delta_{kl}+\mu^{(0)}(\delta_{jk}\delta_{jl}+\delta_{jk}\delta_{il})+t_{jl}(x)\delta_{ik}\label{eq:Lame constants}
\end{equation}
be a constant elasticity four tensor with residual stress satisfying
(\ref{eq:Strong convexity}) and $\rho\equiv1$ be a constant density
in a homogeneous medium. In this case, the constant elasticity system
with residual stress (\ref{eq:Main Equations}) can be reduced to
the following boundary value problem 
\begin{equation}
\begin{cases}
\mathcal{L}_{0}^{R}u_{0}+\kappa^{2}u_{0}=0 & \mbox{ in }\Omega,\\
\mathcal{N}_{\mathcal{C}^{(0)}}u_{0}=\phi & \mbox{ on }\partial\Omega,
\end{cases}\label{eq:Homogeneous Equation}
\end{equation}
where $\mathcal{L}_{0}^{R}u:=\nabla\cdot(\mathcal{C}^{(0)}\nabla u_{0})$
is the second order elasticity operator with residual stress and the
boundary traction $\mathcal{N}_{\mathcal{C}^{(0)}}u_{0}=\nu\cdot(\mathcal{C}^{(0)}\nabla u_{0})$.
More precisely, by (\ref{eq:Symmetric}) and (\ref{eq:zero traction boundary}),
componentwisely, $u_{0}=(u_{0,1},u_{0,2},\cdots,u_{0,N})$, it is
easy to see 
\[
\sum_{j,k,l=1}^{N}\nu_{j}\left(t_{jl}(x)\delta_{ik}\right)\dfrac{\partial u_{0,k}}{\partial x_{l}}=\sum_{j,k,l=1}^{N}\left(t_{lj}(x)\nu_{j}\right)\delta_{ik}\dfrac{\partial u_{0,k}}{\partial x_{l}}=0\mbox{ for }x\in\partial\Omega,
\]
for $i=1,2,\cdots,N$ and $N=2,3$, or equivalently, $\nu\cdot(T\nabla u_{0})=0$
on $\partial\Omega$, so we can represent the boundary traction for
the elasticity system with residual stress by 
\[
\mathcal{N}_{C^{(0)}}u_{0}=\begin{cases}
2\mu^{(0)}\dfrac{\partial u_{0}}{\partial\nu}+\lambda^{(0)}\nu\nabla\cdot u_{0}+\mu^{(0)}\nu^{T}(\partial_{2}u_{0,1}-\partial_{1}u_{0,2}) & \mbox{ when }N=2,\\
2\mu^{(0)}\dfrac{\partial u_{0}}{\partial\nu}+\lambda^{(0)}\nu\nabla\cdot u_{0}+\mu^{(0)}\nu\times(\nabla\times u_{0}) & \mbox{ when }N=3,
\end{cases}
\]
where $\nu$ is a unit outer normal on $\partial\Omega$ and $\nu^{T}=(-\nu_{2},\nu_{1})\perp\nu$
and $u_{0}=(u_{0,1},u_{0,2})$ when $N=2$. In this paper, we can
regard the medium $\left\{ \Omega;\mathcal{C}^{(0)},1\right\} $ as
the \textit{free} reference space in our study about the invisible
cloaking for the elasticity system with residual stress in this work.

Indeed, we have the following facts about the invariance of coordinates
transformation for the elasticity system with residual stress. Let
$\widetilde{x}=F(x):\Omega\to\widetilde{\Omega}$ be an orientation-preserving,
bi-Lipschitz mapping, then we have the following push-forward relations
of $\mathcal{C}$ and $\rho$ defined by 
\begin{eqnarray*}
\widetilde{\mathcal{C}} & := & F_{*}\mathcal{C}=\widetilde{C}_{iqkp}(\widetilde{x})=\left.\dfrac{1}{\det M}\left\{ \sum_{j,l=1}^{N}C_{ijkl}\dfrac{\partial\widetilde{x_{p}}}{\partial x_{l}}\dfrac{\partial\widetilde{x_{q}}}{\partial x_{j}}\right\} \right|_{x=F^{-1}(\widetilde{x})},\\
\widetilde{\rho} & := & F_{*}\rho=\left.\left(\dfrac{\rho}{\det M}\right)\right|_{x=F^{-1}(\widetilde{x})},
\end{eqnarray*}
where $M=\left(\dfrac{\partial\widetilde{x_{i}}}{\partial x_{j}}\right)_{i,j=1}^{N}$.
To simplify the notation, we set $\widetilde{\nabla}=\nabla_{\widetilde{x}}$
and $\left\{ \widetilde{\Omega};\widetilde{\mathcal{C}},\widetilde{\rho}\right\} =F_{*}\left\{ \Omega;\mathcal{C},\rho\right\} $.
Besides, for any $N\times N$ symmetric matrix $(\varepsilon_{ij})$,
we have 
\[
\sum_{i,q,k,p=1}^{N}\widetilde{C}_{iqkp}a_{iq}a_{kp}=\dfrac{1}{\det M}\sum_{i,j,k,l=1}^{N}C_{ijkl}\widetilde{\epsilon}_{ij}\widetilde{\varepsilon}_{kl},
\]
where 
\[
\widetilde{\varepsilon}_{ij}=\sum_{q=1}^{N}\dfrac{\partial\widetilde{x_{q}}}{\partial x_{j}}\varepsilon_{iq},\mbox{ for }i,j=1,2,\cdots,N.
\]
Using the strong convexity condition (\ref{eq:Strong convexity})
and the bi-Lipschitz property of $F$, we obtain 
\begin{equation}
\sum_{i,q,k,p=1}^{N}\widetilde{C}_{iqkp}\widetilde{\varepsilon}_{ij}\widetilde{\varepsilon}_{kl}\geq c_{0}\sum_{i,j=1}^{N}|\widetilde{\varepsilon}_{ij}|^{2}\geq\widetilde{c_{0}}\sum_{i,j=1}^{N}|\varepsilon_{ij}|^{2}\label{eq:Strong convexity after transformation}
\end{equation}
for some $\widetilde{c_{0}}>0$. Moreover, the straightforward calculation
will give that $\widetilde{C}$ is major symmetric but not minor symmetric,
i.e., 
\begin{equation}
\widetilde{C}_{iqkp}=\widetilde{C}_{kpiq},\mbox{ but }\widetilde{C}_{iqkp}\neq\widetilde{C}_{qikp}\mbox{ for any }i,q,k,p=1,2,\cdots,N.\label{eq:transformed symmetry}
\end{equation}
The relations (\ref{eq:Symmetry for residual stress}) and (\ref{eq:transformed symmetry})
make the elastic tensor with residual stress is \textit{invariant}
under the change of coordinates. Therefore, the elasticity system
with residual stress is invariant under the coordinate transformation.
More explicitly, we have the following lemma, which was proved in
\cite{hu2015nearly}.
\begin{lem}
\label{lem:Invarant NtD}\cite{hu2015nearly} The following relation
holds: $u\in H^{1}(\Omega)^{N}$ is a solution to 
\begin{equation}
\nabla\cdot(\mathcal{C}\nabla u)+\kappa^{2}\rho u=0\mbox{ in }\Omega,\label{eq:1}
\end{equation}
if and only if $\widetilde{u}=u\circ F^{-1}\in H^{1}(\widetilde{\Omega})$
is a solution to 
\begin{equation}
\widetilde{\nabla}\cdot(\widetilde{\mathcal{C}}\widetilde{\nabla}\widetilde{u})+\kappa^{2}\widetilde{\rho}\widetilde{u}=0\mbox{ in }\widetilde{\Omega}.\label{eq:2}
\end{equation}
Moreover, if $F=\mathrm{Identity}$ on $\partial\Omega$, then 
\[
\Lambda_{\mathcal{C},\rho}=\Lambda_{\widetilde{\mathcal{C}},\widetilde{\rho}},
\]
where $\Lambda_{\mathcal{C},\rho}$ and $\Lambda_{\widetilde{\mathcal{C}},\widetilde{\rho}}$
are the NtD maps associated to (\ref{eq:1}) and (\ref{eq:2}), respectively.
\end{lem}
The study on the invisible cloaking has been attracted the most attention
among theoretical and practical point of views. A region is cloaked
if its contents together with the cloak are invisible to wave detection.
A transformational cloaking using the invariance properties of the
conductivity equation was discovered by Greenleaf, Lassas and Uhlmann
\cite{greenleaf2003anisotropic,greenleaf2003nonuniqueness}. In \cite{leonhardt2006optical,pendry2006controlling},
the authors used the electromagnetic (EM) waves to make objects invisible.
The basic idea of cloaking is using the invariance of a coordinate
transformation for specific systems, such as conductivity, acoustic,
electromagnetic, and elasticity systems. We refer readers to the survey
articles \cite{chen2010acoustic,greenleaf2009cloaking,greenleaf2009invisibility,norris2008acoustic,uhlmann2014seenunseen}
in physics and mathematics literature for the cloaking theory and
their developments. The perfect cloaking can be obtained by the one-point-blowup
construction, but it will induce the transformed medium to be singular.
The singular structure arises a great challenge for mathematical analysis
and practical purposes for the cloaking theory. 

In order to handle the singular structure from the perfect cloaking
constructions, Greenleaf, Kurylev, Lassas and Uhlmann \cite{greenleaf2007full}
developed a truncation of singularities methods, which is called the
\textit{nearly cloaking} (or approximate cloaking) techniques. More
specifically, this nearly cloaking method is via a special singular
double coating to defeat the singular structure. In \cite{greenleaf2007improvement,greenleaf2008isotropic,ruan2007ideal},
the authors utilized this truncation of singularities methods to approach
the nearly cloaking theory. On the other hand, there are many researchers
have considered another methods, called the small-inclusion-blowup
construction, which regularized the singular medium which was induced
by the one-point-blowup construction for the perfect cloaking. The
small-inclusion-blowup method was studied by many researchers: \cite{ammari2013conductivity,kohn2008cloaking}
for the conductivity model, \cite{ammari2012enhancement,ammari2013Helmholtz,deng2015regularized,kohn2010cloaking,li2015regularized,liu2009virtual,liu2013near,liu2013enhanced}
for the acoustic model, \cite{ammari2013Maxwell,bao2014nearly,baojun2014nearly,deng2015full}
for the Maxwell model and \cite{hu2015nearly} for the Lamé model.
In \cite{kocyigit2013regular}, the authors also pointed out that
the truncation of singularity construction and the small-inclusion
construction are equivalent. Especially, we are interested about the
small-inclusion-blowup construction to approach our nearly cloaking
theory for the elasticity system with residual stress. In further,
there are some physics and mathematics literature for the elastic
cloaking theory, such as \cite{brun2009achieving,diatta2013cloaking,diatta2014controlling,milton2006cloaking,norris2008acoustic}. 

However, for the isotropic elastic waves, which is governed by the
Lamé system, is not invariance under a coordinate transformation,
i.e., the minor symmetric structure will break after a coordinate
transformation. This phenomena means that any elastic four tensor
with the major and minor symmetric properties, then the transformed
elastic tensor will not be minor symmetric anymore. Fortunately, for
the elastic tensor with residual stress possesses only the major symmetric
property but the minor symmetry breaks, see (\ref{eq:Symmetry for residual stress}),
(\ref{eq:transformed symmetry}). These relations imply that the elasticity
system with residual stress is invariant via a coordinate transformation
and the invariant property makes the elastic waves with residual stress
is more natural than the isotropic elastic waves in the physical sense.
It is hard to build up the cloaking theory for the elasticity system
with residual stress because there are no layer potential theories
for such system. Instead, we will derive suitable global estimates
for the solutions of the elasticity system with residual stress under
appropriate boundary traction conditions.

The paper is organized as follows. In Section 2, we introduce the
blowup constructions to achieve the perfect cloaking for our residual
stress model. In Section 3, we explain how to avoid the singular structure
of the elastic tensor with the residual stress by using the regularization
technique, which is the small-inclusion-blowup method. Meanwhile,
in order to avoid the non-uniqueness for the NtD map after a coordinate
transformation, we state our main results of the nearly cloaking theory
for the elasticity system with residual stress by demonstrating the
\textit{lossy layer} technique. In Section 4, we provide useful tools
and global estimates for this second order anisotropic elasticity
systems and complete the proof of our main theorem. Finally, in Appendix,
we give a glimpse review of the layer potential theory for the Lamé
system, which will be used in the proof of our main theorem.\\
\\
\textbf{Acknowledgments.} The author would like to thank Prof. Gunther
Uhlmann for suggesting this problem and providing useful advises.
Y.H. is a postdoctoral fellow at Institute for Advanced Study (IAS),
Jockey Club, HKUST, Hong Kong.

\section{Perfect cloaking for the elasticity system with residual stress }

In this section, we will construct the perfect cloaking for the elasticity
system with residual stress by using the one-point-blowup approach
to achieve our goal. For $N=2,3,$ let $\Omega\subset\mathbb{R}^{N}$
and $D\Subset\Omega$ be bounded and connected $C^{\infty}$-smooth
domains. Moreover, we also assume that $\Omega\backslash\overline{D}$
is connected and $D$ contains the origin in $\mathbb{R}^{N}$. For
any $h>0$, $D_{h}=\{hx:x\in D\}$ and let 
\[
\left\{ D_{\frac{1}{2}};\mathcal{C}^{(a)},\rho^{(a)}\right\} 
\]
be the target medium with $D_{\frac{1}{2}}$ denoting the region which
we want to cloak. 
\begin{defn}
We say the medium $\left\{ \Omega;\mathcal{C},\rho\right\} $ to be
\textit{regular} if $\mathcal{C}$ satisfies the strong convexity
condition (\ref{eq:Strong convexity}) and the major symmetric condition
(\ref{eq:Symmetry for residual stress}).
\end{defn}
In this section, we assume the medium $\left\{ D_{\frac{1}{2}};\mathcal{C}^{(a)},\rho^{(a)}\right\} $
to be arbitrary but regular and consider 
\[
\left\{ \Omega\backslash\overline{D_{\frac{1}{2}}};\mathcal{C}^{(c)},\rho^{(c)}\right\} 
\]
as a suitable layer of elastic medium, which is designed to be the
cloaking medium. Let 
\begin{equation}
\left\{ \Omega;\mathcal{C},\rho\right\} =\begin{cases}
\left\{ \Omega\backslash\overline{D_{\frac{1}{2}}};\mathcal{C}^{(c)},\rho^{(c)}\right\}  & \mbox{ in }\Omega\backslash\overline{D_{\frac{1}{2}}},\\
\left\{ D_{\frac{1}{2}};\mathcal{C}^{(a)},\rho^{(a)}\right\}  & \mbox{ in }D_{\frac{1}{2}},
\end{cases}\label{eq:perfect medium}
\end{equation}
be the medium occupying $\Omega$. We can define the associated NtD
map
\begin{equation}
\Lambda_{\mathcal{C},\rho}:H^{-1/2}(\partial\Omega)^{N}\to H^{1/2}(\partial\Omega)^{N}\label{eq:perfect NtD map}
\end{equation}
in the medium (\ref{eq:perfect medium}). In addition, the medium
$D_{\frac{1}{2}}$ and $\Omega\backslash\overline{D_{\frac{1}{2}}}$
will be specified appropriately in the forthcoming discussions whenever
it is necessary. If $-\kappa^{2}$ be not an eigenvalue of the elliptic
operator $\mathcal{L}_{0}^{R}$ with the zero boundary traction condition,
then there exists a unique solution $u_{0}\in H^{1}(\Omega)^{N}$
to 
\[
\begin{cases}
\mathcal{L}_{0}^{R}u_{0}+\kappa^{2}u_{0}=0 & \mbox{ in }\Omega,\\
\mathcal{N}_{\mathcal{C}^{(0)}}u_{0}=\phi & \mbox{ on }\partial\Omega,
\end{cases}
\]
and the corresponding NtD map $\Lambda_{0}:H^{-1/2}(\partial\Omega)^{N}\to H^{1/2}(\partial\Omega)^{N}$
is well-defined on the free reference space $\left\{ \Omega;\mathcal{C}^{(0)},1\right\} $
with $\Lambda_{0}\phi=u_{0}|_{\partial\Omega}$.
\begin{defn}
\label{def:perfect cloak}The medium $\left\{ \Omega\backslash\overline{D_{\frac{1}{2}}};\mathcal{C}^{(c)},\rho^{(c)}\right\} $
is said to be \textit{perfect (elastic) cloak} if $\Lambda_{\mathcal{C},\rho}(\phi)=\Lambda_{0}(\phi)$
for any $\phi\in H^{-1/2}(\partial\Omega)^{N}$.
\end{defn}
On the basis of Definition \ref{def:perfect cloak}, the cloaking
layer $\left\{ \Omega\backslash\overline{D_{\frac{1}{2}}};\mathcal{C}^{(c)},\rho^{(c)}\right\} $
makes itself and the target medium $\left\{ D_{\frac{1}{2}};\mathcal{C}^{(a)},\rho^{(a)}\right\} $
``invisible'' by the exterior elastic wave measurements. Now, we
want to analyze the singular structure from the mathematical viewpoint.
For $R>0$, let $B_{R}=\{x:|x|<R\}$ be a ball centered at the origin
and radius $R>0$. In this section, we take $\Omega=B_{2}$ and $D_{\frac{1}{2}}=B_{1}$
to demonstrate the analysis for the singular structures of the perfect
cloaking. Consider the singular transformation $F_{0}:B_{2}\backslash\{0\}\to B_{2}\backslash\overline{B_{1}}$
by 
\begin{equation}
F_{0}(x)=(1+\frac{|x|}{2})\frac{x}{|x|},\mbox{ for }0<|x|\leq2.\label{eq:singular transformation}
\end{equation}
It is easy to see that $F_{0}$ blows up at $0$ in the reference
space to $B_{1}$ and $F_{0}$ maps $B_{2}\backslash\{0\}$ to $B_{2}\backslash\overline{B_{1}}$
with $F_{0}|_{\partial B_{2}}=$Identity. Denoting $y:=F_{0}(x)$,
the transformed medium in $\left\{ B_{2}\backslash\overline{B_{1}};\mathcal{C}^{(c)},\rho^{(c)}\right\} $
can be represented as 
\[
\mathcal{C}^{(c)}(y)=(F_{0})_{*}(\mathcal{C}^{(0)}(x))|_{x=F_{0}^{-1}(y)}\mbox{ and }\rho^{(c)}(y)=(F_{0})_{*}(1)(x)|_{x=F_{0}^{-1}(y)},
\]
where $(F_{0})_{*}$ is the push-forward of $F_{0}$. Now, we consider
the boundary value problem (\ref{eq:Main Equations}) in $\Omega=B_{2}$
with 
\[
\left\{ B_{2};\mathcal{C},\rho\right\} =\begin{cases}
\left\{ B_{2}\backslash\overline{B_{1}};\mathcal{C}^{(0)},\rho^{(0)}\right\}  & \mbox{ in }B_{2}\backslash\overline{B_{1}},\\
\left\{ B_{1};\mathcal{C}^{(a)},\rho^{(a)}\right\}  & \mbox{ in }B_{1},
\end{cases}
\]
and the corresponding NtD map is $\Lambda_{\mathcal{C}.\rho}$, which
was given by (\ref{eq:perfect NtD map}), then we have 
\begin{prop}
(Perfect cloaking) We have 
\[
\Lambda_{0}=\Lambda_{\mathcal{C},\rho}\mbox{ on }\partial B_{2}.
\]
\end{prop}
\begin{proof}
Let $\widetilde{\Lambda_{0}}$ be the NtD map associated with the
medium $\left\{ B_{2}\backslash\{0\};\mathcal{C}^{(0)},1\right\} $.
By Lemma \ref{lem:Invarant NtD}, we have 
\[
\widetilde{\Lambda_{0}}=\Lambda_{\mathcal{C},\rho}\mbox{ on }\partial B_{2}.
\]
Note that the inhomogeneity of $\left\{ \Omega\backslash\{0\};\mathcal{C}^{(0)},1\right\} $
is only supported in a singular point $\{0\}$, then we have $\widetilde{\Lambda_{0}}=\Lambda_{0}$
and complete the proof.
\end{proof}
The above result for the perfect elastic cloaking is easy and intuitive.
Note that the singular transformation $F_{0}$ will force the elastic
tensor $\mathcal{C}^{(c)}$ not to satisfy the strong convexity condition
(\ref{eq:Strong convexity}). Indeed, Hu and Liu \cite{hu2015nearly}
gave explicit analysis for the singular structures of the transformed
elastic parameters $\mathcal{C}^{(c)}$ and $\rho^{(c)}$, so we omit
the arguments and refer readers to Section 3 in \cite{hu2015nearly}
for further and complete discussions. 

On the other hand, from practical aspects, we can calculate the elastic
tensor with residual stress under the radial case by using the polar
coordinates $(r,\theta)$ in 2D and the spherical coordinates $(r,\theta,\varphi)$
in 3D. In order to calculate the transformed elastic parameter $\mathcal{C}^{(c)}$
(\ref{eq:Lame constants}) explicitly, for convenience, we simplify
indexes by $1\to r$, $2\to\theta$, $r\to\varphi$ and use the \textit{Voigt}
notation for tensor indices, which are 
\begin{align*}
 & 11\to1,\mbox{ }22\to2\mbox{ }33\to3,\\
 & 23,32\to4,\mbox{ }13,31\to5,\mbox{ }12,21\to6,\mbox{ }
\end{align*}
then the residual stress tensor (\ref{eq:Lame constants}) can be
written as 
\[
C_{ab}=\Lambda_{ab}+\mathcal{T}_{ab}\mbox{ for }1\leq a,b\leq N\mbox{ and }N=2,3.
\]
More explicitly, when $N=2$, 
\begin{equation}
\Lambda_{ab}=\left(\begin{array}{ccc}
\lambda^{(0)}+2\mu^{(0)} & \lambda^{(0)} & 0\\
\lambda^{(0)} & \lambda^{(0)}+2\mu^{(0)} & 0\\
0 & 0 & \mu^{(0)}
\end{array}\right)\label{eq:2D isotropic elastic tensor}
\end{equation}
and 
\begin{equation}
\mathcal{T}_{ab}=\left(\begin{array}{ccc}
t_{11} & 0 & t_{12}\\
0 & t_{22} & t_{12}\\
t_{12} & t_{12} & t_{11}+t_{22}
\end{array}\right)(x).\label{eq:2D residual stress}
\end{equation}
When $N=3$, 
\[
\Lambda_{ab}=\left(\begin{array}{cccccc}
\lambda^{(0)}+2\mu^{(0)} & \lambda^{(0)} & \lambda^{(0)} & 0 & 0 & 0\\
\lambda^{(0)} & \lambda^{(0)}+2\mu^{(0)} & \lambda^{(0)} & 0 & 0 & 0\\
\lambda^{(0)} & \lambda^{(0)} & \lambda^{(0)}+2\mu^{(0)} & 0 & 0 & 0\\
0 & 0 & 0 & \mu^{(0)} & 0 & 0\\
0 & 0 & 0 & 0 & \mu^{(0)} & 0\\
0 & 0 & 0 & 0 & 0 & \mu^{(0)}
\end{array}\right)
\]
and 
\[
\mathcal{T}_{ab}=\left(\begin{array}{cccccc}
t_{11} & 0 & 0 & 0 & t_{13} & t_{12}\\
0 & t_{22} & 0 & t_{23} & 0 & t_{12}\\
0 & 0 & t_{33} & t_{23} & t_{13} & 0\\
0 & t_{23} & t_{23} & t_{22}+t_{33} & t_{12} & t_{13}\\
t_{13} & 0 & t_{13} & t_{12} & t_{11}+t_{33} & t_{23}\\
t_{12} & t_{12} & 0 & t_{13} & t_{23} & t_{11}+t_{22}
\end{array}\right)(x).
\]
Furthermore, let $DF_{0}$ be the Jacobian matrix of $F_{0}$, for
the radial case $r=|y|$, we have 
\begin{equation}
DF_{0}=\left(\begin{array}{cc}
\frac{1}{2} & 0\\
0 & \frac{r}{2(r-1)}I_{N-1}
\end{array}\right),\label{eq:polar coor representation}
\end{equation}
where $I_{N-1}$ is an $(N-1)\times(N-1)$ identity matrix for $N=2,3$
and for any $r\in(1,2)$. Now, we can use the following transformation
formula 
\begin{equation}
\left.\mathcal{C}^{(c)}(y)=\dfrac{DF_{0}(x)\XBox\mathcal{C}^{(0)}(x)\XBox DF_{0}^{t}(x)}{\det DF_{0}(x)}\right|_{x=F_{0}^{-1}(y)}\mbox{ (}t\mbox{ for transpose)},\label{eq:transformation formula}
\end{equation}
where $\XBox$ is the multiplication between a matrix and fourth rank
tensor and 
\[
\det DF_{0}=\begin{cases}
\frac{r}{4(r-1)} & \mbox{ for }N=2,\\
\frac{r^{2}}{8(r-1)^{2}} & \mbox{ for }N=3.
\end{cases}
\]
More explicitly, for $\mathcal{C}^{(c)}(y)=(C_{ijkl}^{(c)}(y))_{i,j,k,l=1}^{N}$,
(\ref{eq:transformation formula}) is equivalent to 
\begin{align}
 & \left(\begin{array}{cccc}
C_{i1k1}^{(c)} & C_{i1k2}^{(c)} & \cdots & C_{i1kN}^{(c)}\\
C_{i2k1}^{(c)} & C_{i2k2}^{(c)} & \cdots & C_{i2kN}^{(c)}\\
\vdots & \vdots & \cdots & \vdots\\
C_{iNk1}^{(c)} & C_{iNk2}^{(c)} & \cdots & C_{iNkN}^{(c)}
\end{array}\right)\label{eq:explicit matrix transform}\\
= & \dfrac{1}{\det DF_{0}}DF_{0}\left(\begin{array}{cccc}
C_{i1k1}^{(0)} & C_{i1k2}^{(0)} & \cdots & C_{i1kN}^{(0)}\\
C_{i2k1}^{(0)} & C_{i2k2}^{(0)} & \cdots & C_{i2kN}^{(0)}\\
\vdots & \vdots & \cdots & \vdots\\
C_{iNk1}^{(0)} & C_{iNk2}^{(0)} & \cdots & C_{iNkN}^{(0)}
\end{array}\right)DF_{0}^{t}\nonumber 
\end{align}
for $i,k=1,2,\cdots,N$. Making use of the polar coordinate, (\ref{eq:2D isotropic elastic tensor}),
(\ref{eq:2D residual stress}), (\ref{eq:polar coor representation})
and (\ref{eq:explicit matrix transform}), we can represent the transformed
tensor $\mathcal{C}^{(c)}(y)=\mathcal{C}^{(c)}(r)$ for $N=2$ in
the following by straightforward calculation, there are twelve nontrivial
entries for the residual stress systems: 
\begin{align*}
 & C_{rrrr}^{(c)}=(\lambda^{(0)}+2\mu^{(0)}+t_{11}(F_{0}^{-1}(y)))\frac{r-1}{r},\mbox{ }C_{\theta r\theta r}^{(c)}=(\mu^{(0)}+t_{11}(F_{0}^{-1}(y)))\frac{r}{r-1},\\
 & C_{rrr\theta}^{(c)}=C_{r\theta rr}^{(c)}=\frac{r-1}{r}t_{12}(F_{0}^{-1}(y)),\mbox{ }C_{\theta r\theta\theta}^{(c)}=C_{\theta\theta\theta r}^{(c)}=\frac{r}{r-1}t_{12}(F_{0}^{-1}(y)),\\
 & C_{rr\theta\theta}^{(c)}=C_{\theta\theta rr}^{(0)}=\lambda^{(0)},\mbox{ }C_{r\theta\theta r}^{(c)}=C_{\theta rr\theta}^{(c)}=\mu^{(0)},\\
 & C_{r\theta r\theta}^{(c)}=(\mu^{(0)}+t_{22}(F_{0}^{-1}(y)))\frac{r-1}{r},\mbox{ }C_{\theta\theta\theta\theta}^{(c)}=(\lambda^{(0)}+2\mu^{(0)}+t_{22}(F_{0}^{-1}(y)))\frac{r}{r-1}.
\end{align*}
Note that the strong convexity condition (\ref{eq:Strong convexity})
for $N=2$, we have 
\[
2\lambda^{(0)}+2\mu^{(0)}+t_{22}(F_{0}^{-1}(y))\geq c_{0}>0\mbox{ for all }y\in B_{2}\backslash\overline{B_{1}}.
\]
Hence, it is easy to see that 
\[
\lambda^{(0)}+2\mu^{(0)}+t_{22}(F_{0}^{-1}(y))>0\mbox{ and }\mu^{(0)}+t_{11}(F_{0}^{-1}(y))>0\mbox{ for any }y\in B_{2}\backslash\overline{B_{1}}.
\]
and these imply that $C_{\theta r\theta r}^{(c)},C_{\theta\theta\theta\theta}^{(c)}$
have singularities on $\partial B_{1}$. On the other hand, by $T\in W^{2,\infty}(B_{2})$,
we can derive that $C_{rrrr}^{(c)},C_{r\theta r\theta}^{(c)}$ vanish
on $\partial B_{1}$ and. Similar calculation will be valid for the
transformed tensor $\mathcal{C}^{(c)}$ for $N=3$ under the spherical
coordinates, so we skip details here. For more discussions of the
transformed elastic tensors, we refer readers to \cite{brun2009achieving,diatta2014controlling,hu2015nearly}.

Based on the blowup construction, we can obtain the perfect cloaking
for the elasticity system with residual stress. However, the elastic
materials possess singular structure, so the one-point-blowup construction
is not useful for the mathematical analysis and practical applications.

\section{Nearly cloaking construction and main result}

The perfect cloaking for the elasticity system with residual stress
produces singular parameters. In order to avoid singularities, we
introduce the nearly cloaking in this section via the small-inclusion-blowup
construction and the \textit{lossy layer }techniques as follows.

\subsection{Small-inclusion-blowup and non-uniqueness for NtD maps}

In the beginning of this section, we demonstrate our nearly-cloaking
algorithm by considering the radial case. For small $h>0$, let $\widehat{F_{h}}:B_{2}\backslash\overline{B_{h}}\to B_{2}\backslash\overline{B_{1}}$
and $\widehat{F_{h}}|_{\partial B_{2}}=$Identity be a transformation
given by 
\[
\widehat{F_{h}}(x):=\left(\frac{2-2h}{2-h}+\frac{|x|}{2-h}\right)\frac{x}{|x|}.
\]
Consider the following transformation 
\begin{equation}
F_{h}(x):=\begin{cases}
\widehat{F_{h}}(x) & \mbox{ for }h\leq|x|\leq2,\\
\frac{x}{h} & \mbox{ for }|x|<h.
\end{cases}\label{eq:approximate transformation}
\end{equation}
It is obvious that that and $F_{h}:B_{2}\to B_{2}$ is a bi-Lipschitz,
orientation-preserving and $F_{h}|_{\partial B_{2}}=$Identity. By
using the same small-inclusion-blowup construction as in \cite{hu2015nearly},
we have the following form: 
\begin{equation}
\left\{ B_{2};\mathcal{C},\rho\right\} =\begin{cases}
\left\{ B_{2}\backslash\overline{B_{1}};\mathcal{C}_{h}^{(c)},\rho_{h}^{(c)}\right\}  & \mbox{ in }B_{2}\backslash\overline{B_{1}},\\
\left\{ B_{1};\mathcal{C}^{(a)},\rho^{(a)}\right\}  & \mbox{ in }B_{1},
\end{cases}\label{eq:Near Cloaking 1}
\end{equation}
with the cloaking medium given by 
\[
\mathcal{C}_{h}^{(c)}:=(F_{h})_{*}(\mathcal{C}^{(0)})(x)|_{x=F_{h}^{-1}(y)},\mbox{ }\rho_{h}^{(c)}:=(F_{h})_{*}(1)(x)|_{x=F_{h}^{-1}(y)},\mbox{ }y\in B_{2}\backslash\overline{B_{1}}.
\]
Let $\Lambda_{h}$ be the NtD map associated with the configuration
(\ref{eq:Near Cloaking 1}). It is easy to see that $F_{h}\to F_{0}$
as $h\to0$, where $F_{0}$ is the singular transformation given by
(\ref{eq:singular transformation}), so we expect $\Lambda_{h}\to\Lambda_{0}$
as $h\to0$. However, in the medium (\ref{eq:Near Cloaking 1}), NtD
map $\Lambda_{h}$ may not be well-defined and we will explain more
in the following. For $h>0$ small, we write 
\[
\left\{ B_{2};\widetilde{\mathcal{C}},\widetilde{\rho}\right\} :=(F_{h}^{-1})_{*}\left\{ B_{2};\mathcal{C},\rho\right\} =\begin{cases}
\left\{ B_{2}\backslash\overline{B_{h}};\mathcal{C}^{(0)},1\right\}  & \mbox{ in }B_{2}\backslash\overline{B_{h}},\\
\left\{ B_{h};\widetilde{\mathcal{C}}^{(a)},\widetilde{\rho}^{(a)}\right\}  & \mbox{ in }B_{h},
\end{cases}
\]
where $\left\{ B_{h};\widetilde{\mathcal{C}}^{(a)},\widetilde{\rho}^{(a)}\right\} =(F_{h}^{-1})_{*}\left\{ B_{1};\mathcal{C}^{(a)},\rho^{(a)}\right\} $. 

Since the target medium $\left\{ B_{1};\mathcal{C}^{(a)},\rho^{(a)}\right\} $
is arbitrary but regular, by (\ref{eq:Strong convexity after transformation}),
the small inclusion $\left\{ B_{h};\widetilde{\mathcal{C}}^{(a)},\widetilde{\rho}^{(a)}\right\} $
is arbitrary but regular, which means we avoid the singular structure
by considering the transformation $F_{h}$. By Lemma \ref{lem:Invarant NtD}
again, $\widetilde{\Lambda_{h}}=\Lambda_{h}$, where $\widetilde{\Lambda_{h}}$
is the transformed NtD map associated to the medium $\left\{ B_{2};\widetilde{\mathcal{C}},\widetilde{\rho}\right\} $
and $\widetilde{\Lambda_{h}}$ may not be well-defined. 

Note that the nearly cloaking construction (\ref{eq:Near Cloaking 1})
via the transformation (\ref{eq:approximate transformation}) might
fail because of the corresponding NtD map $\widetilde{\Lambda_{h}}$
may not be well-defined. More specifically, for $x\in\mathbb{R}^{N}$,
$N=2,3$ and for any $0<r_{0}<r_{1}$, we can consider the following
elasticity system with residual stress 
\begin{equation}
\begin{cases}
\nabla\cdot(\mathcal{C}^{(0)}\nabla u_{1})+\kappa^{2}\rho_{1}u_{1}=0 & \mbox{ in }r_{0}<|x|<r_{1},\\
\nabla\cdot(\mathcal{C}^{(0)}\nabla u_{0})+\kappa^{2}\rho_{0}u_{0}=0 & \mbox{ in }|x|<r_{0},\\
\mathcal{N}_{\mathcal{C}^{(0)}}u_{1}=0 & \mbox{ on }|x|=r_{1},\\
u_{0}=u_{1}\mbox{ and }\mathcal{N}_{\mathcal{C}^{(0)}}u_{0}=\mathcal{N}_{\mathcal{C}^{(0)}}u_{1} & \mbox{ on }|x|=r_{0},
\end{cases}\label{eq:Transmission problem}
\end{equation}
where $\mathcal{C}^{(0)}$ is the residual stress tensor defined by
(\ref{eq:Lame constants}), $\kappa>0$ is a fixed frequency. By varying
two positive constants $\rho_{0},\rho_{1}$, then there are two cases
might happen: 
\begin{enumerate}
\item \label{enu:Case 1}\textit{Non-resonance effect}: For some constants
$\rho_{0},\rho_{1}>0$, there exists a unique solution $(u_{0},u_{1})=(0,0)$
to (\ref{eq:Transmission problem}), then the corresponding NtD map
on $|x|=r_{1}$ is well-defined (since $(u_{0},u_{1})=(0,0)$ is a
trivial solution to (\ref{eq:Transmission problem})).
\item \label{enu:Case 2}\textit{Resonance effect}: For some constants $\rho_{0},\rho_{1}>0$,
there exists a nontrivial solution $(u_{0},u_{1})\neq(0,0)$ to (\ref{eq:Transmission problem}),
which implies the corresponding NtD map on $|x|=r_{1}$ is not well-defined. 
\end{enumerate}
The resonance effect will make the corresponding NtD map be not well-defined,
which causes the nearly cloaking construction to fail under this situation.
However, no matter what cases (resonance or non-resonance effects)
occurs, we can make use of the \textit{lossy layer} method to guarantee
the corresponding NtD map on the boundary to be well-defined. Thus,
we can make the nearly cloaking for the elasticity system with residual
stress can be constructed successfully. Therefore, the lossy layer
technique is a necessary method to make our near-cloaking construction
work.

\subsection{Lossy layer techniques}

In order to overcome the resonance effects, we introduce a certain
damping mechanism, which means that we build a lossy layer between
the cloaking region and the cloaked area. Recall that we have assumed
$D\Subset\Omega$ be $C^{\infty}$-smooth domains in $\mathbb{R}^{N}$
with $D$ containing the origin and $\Omega\backslash\overline{D}$
is connected for $N=2,3$. Let $h>0$ be small, we consider the orientation-preserving
and bi-Lipschitz map $F_{h}$ with 
\[
\widehat{F_{h}}:\overline{\Omega}\backslash D_{h}\to\overline{\Omega}\backslash D\mbox{ and }\widehat{F_{h}}|_{\partial\Omega}=\mbox{Identity},
\]
where $D_{h}=\{hx:x\in D\}$ and define the mapping  $F$ by 
\[
F_{h}(x)=\begin{cases}
\widehat{F_{h}}(x) & \mbox{ for }x\in\overline{\Omega}\backslash D_{h},\\
\frac{x}{h} & \mbox{ for }x\in D_{h}.
\end{cases}
\]
It is easy to see that $F_{h}:\Omega\to\Omega$ is an orientation-preserving
and bi-Lipschitz map with $F_{h}|_{\partial\Omega}=$Identity. 

As in Section 2, $\left\{ D_{\frac{1}{2}};\mathcal{C}^{(a)},\rho^{(a)}\right\} $
is the medium that we want to cloak. Consider the elastic medium with
residual stress as follows: 
\begin{equation}
\{\Omega;\mathcal{C},\rho\}=\begin{cases}
\left\{ \Omega\backslash\overline{D_{\frac{1}{2}}};\mathcal{C}^{(c)},\rho^{(c)}\right\}  & \mbox{ in }\Omega\backslash\overline{D_{\frac{1}{2}}},\\
\left\{ D_{\frac{1}{2}};\mathcal{C}^{(a)},\rho^{(a)}\right\}  & \mbox{ in }D_{\frac{1}{2}}.
\end{cases}\label{eq:Near cloaking with lossy}
\end{equation}
where 
\begin{equation}
\{\Omega\backslash\overline{D_{\frac{1}{2}}};\mathcal{C}^{(c)},\rho^{(c)}\}=\begin{cases}
\left\{ \Omega\backslash\overline{D};\mathcal{C}^{(1)},\rho^{(1)}\right\}  & \mbox{ in }\Omega\backslash\overline{D},\\
\left\{ D\backslash\overline{D_{\frac{1}{2}}};\mathcal{C}^{(2)},\rho^{(2)}\right\}  & \mbox{ in }D\backslash\overline{D_{\frac{1}{2}}},
\end{cases}\label{eq:Near Cloaking 2}
\end{equation}
and 
\begin{eqnarray}
\left\{ \Omega\backslash\overline{D};\mathcal{C}^{(1)},\rho^{(1)}\right\}  & = & (F_{h})_{*}\left\{ \Omega\backslash\overline{D_{h}};\mathcal{C}^{(0)},1\right\} ,\label{eq:Near Cloaking 3}\\
\left\{ D\backslash\overline{D_{\frac{1}{2}}};\mathcal{C}^{(2)},\rho^{(2)}\right\}  & = & (F_{h})_{*}\left\{ D_{h}\backslash\overline{D_{\frac{h}{2}}};\widetilde{\mathcal{C}}^{(2)},\widetilde{\rho}^{(2)}\right\} .\label{eq:Near Cloaking 4}
\end{eqnarray}
Note that the elastic medium in $D_{h}\backslash\overline{D_{\frac{h}{2}}}$
is designed by 
\begin{equation}
\left\{ D_{h}\backslash\overline{D_{\frac{h}{2}}};\widetilde{\mathcal{C}}^{(2)},\widetilde{\rho}^{(2)}\right\} ,\mbox{ }\widetilde{\mathcal{C}}^{(2)}=\gamma h^{2+\delta}\widetilde{\mathcal{C}}^{(0)}\mbox{ and }\widetilde{\rho}^{(2)}=\alpha+i\beta,\label{eq:Near Cloaking 5}
\end{equation}
where $\alpha,\beta,\gamma$ and $\delta$ are given positive constants
and $\left\{ D_{h}\backslash\overline{D_{\frac{h}{2}}};\widetilde{\mathcal{C}}^{(2)},\widetilde{\rho}^{(2)}\right\} $
is our critical lossy layer. Note the number $\beta>0$ is the \textit{damping}
parameter of the medium, which guarantees the well-defined property
for the NtD map with respect the medium $\left\{ \Omega;\mathcal{C},\rho\right\} $
defined by (\ref{eq:Near cloaking with lossy})-(\ref{eq:Near Cloaking 5}).
Note that the selection of parameters in (\ref{eq:Near Cloaking 5})
plays the crucial role for the construction of the lossy layer.

Now we consider the boundary value problem 
\begin{equation}
\begin{cases}
\nabla\cdot(\mathcal{C}\nabla u)+\kappa^{2}\rho(x)u=0 & \mbox{ in }\Omega,\\
\mathcal{N}_{\mathcal{C}}u=\phi\in H^{-\frac{1}{2}}(\partial\Omega)^{3} & \mbox{ on }\partial\Omega,
\end{cases}\label{eq:Near cloaking BVP}
\end{equation}
where $\left\{ \Omega;\mathcal{C},\rho\right\} $ was introduced by
(\ref{eq:Near cloaking with lossy})-(\ref{eq:Near Cloaking 5}). 

We can state our main result for the nearly cloaking of the elasticity
system with residual stress.
\begin{thm}
\label{thm:Main}Suppose $-\kappa^{2}$ is not an eigenvalue of the
elliptic operator $\mathcal{L}_{0}^{R}$ on $\partial\Omega$ with
the zero boundary traction condition $\mathcal{N}_{\mathcal{C}^{(0}}u=0$
on $\partial\Omega$. Let $\Lambda_{\mathcal{C},\rho}$ and $\Lambda_{0}$
the NtD maps of (\ref{eq:Near cloaking BVP}) and (\ref{eq:Homogeneous Equation}),
respectively. Then there exists $h^{(0)}>0$ such that for any $h<h^{(0)}$,
\[
\|\Lambda_{\mathcal{C},\rho}-\Lambda_{0}\|_{\mathscr{L}(H^{-1/2}(\partial\Omega)^{N},H^{1/2}(\partial\Omega)^{N})}\leq Ch,
\]
where $\|\cdot\|_{\mathscr{L}(H^{-1/2}(\partial\Omega)^{N},H^{1/2}(\partial\Omega)^{N})}$
denotes the operator norm from $H^{-1/2}(\partial\Omega)^{N}$ to
$H^{1/2}(\partial\Omega)^{N}$ and $C>0$ is a constant independent
of $h,C^{(a)},\rho^{(a)}$ and $\delta$.
\end{thm}
In order to prove Theorem \ref{thm:Main}, notice that by using Lemma
\ref{lem:Invarant NtD}, then we have 
\begin{equation}
\Lambda_{\mathcal{C},\rho}=\Lambda_{\widetilde{C},\widetilde{\rho}},\label{eq:Invariant NtD map}
\end{equation}
where $\widetilde{C}=(F_{h}^{-1})_{*}$$\mathcal{C}$, $\widetilde{\rho}=(F_{h}^{-1})_{*}\rho$
and 
\[
\{\Omega;\widetilde{\mathcal{C}},\widetilde{\rho}\}=\begin{cases}
\mathcal{C}^{(0)},1 & \mbox{ in }\Omega\backslash\overline{D_{h}},\\
\widetilde{\mathcal{C}}^{(2)},\widetilde{\rho}^{(2)} & \mbox{ in }D_{h}\backslash\overline{D_{\frac{h}{2}}},\\
\widetilde{\mathcal{C}}^{(a)},\widetilde{\rho}^{(a)} & \mbox{ in }D_{\frac{h}{2}},
\end{cases}
\]
with $\widetilde{\mathcal{C}}^{(a)}=(F_{h}^{-1})_{*}\mathcal{C}^{(a)}$,
$\widetilde{\rho}^{(a)}=(F_{h}^{-1})_{*}\rho^{(a)}$ and $\widetilde{\mathcal{C}}^{(2)},\widetilde{\rho}^{(2)}$
were introduced in (\ref{eq:Near Cloaking 5}).

Via Lemma \ref{eq:Variational Equation}, if $u\in H^{1}(\Omega)^{N}$
solves the following elasticity system with residual stress
\[
\begin{cases}
\nabla\cdot(\mathcal{C}\nabla u)+\kappa^{2}\rho u=0 & \mbox{ in }\Omega,\\
\mathcal{N}_{\mathcal{C}}u=\phi & \mbox{ on }\partial\Omega,
\end{cases}
\]
then $\widetilde{u}:=u\circ F_{h}^{-1}$ is the solution of 
\begin{equation}
\begin{cases}
\widetilde{\nabla}\cdot(\widetilde{\mathcal{C}}\widetilde{\nabla}\widetilde{u})+\kappa^{2}\widetilde{\rho}\widetilde{u}=0 & \mbox{ in }\Omega,\\
\mathcal{N}_{\widetilde{\mathcal{C}}}\widetilde{u}=\phi & \mbox{ on }\partial\Omega.
\end{cases}\label{eq:Transfored Equation}
\end{equation}
Now, let $u_{0}$ be a solution of (\ref{eq:Homogeneous Equation}),
by (\ref{eq:Invariant NtD map}), then Theorem \ref{thm:Main} holds
by proving the following theorem.
\begin{thm}
\label{thm:Main 2}Suppose $-\kappa^{2}$ is not an eigenvalue of
the elliptic operator $\mathcal{L}_{0}^{R}$ on $\partial\Omega$
with the zero boundary traction condition $\mathcal{N}_{\mathcal{C}^{(0)}}u=0$
on $\partial\Omega$. Then there exists a constant $h_{0}>0$ such
that for any $h<h_{0}$, 
\begin{equation}
\|\widetilde{u}-u_{0}\|_{H^{1/2}(\partial\Omega)^{N}}\leq Ch\|\phi\|_{H^{-1/2}(\partial\Omega)^{N}},\label{eq:Near cloaking estimte (Main result)}
\end{equation}
where $\widetilde{u}$ and $u_{0}$ are solutions of (\ref{eq:Transfored Equation})
and (\ref{eq:Homogeneous Equation}), respectively and $C>0$ is a
constant independent of $h,\psi,\widetilde{C},\widetilde{\rho}$ and
$\delta$.
\end{thm}

\section{Elliptic estimates and proof of Theorem \ref{thm:Main 2}}

Before proving Theorem \ref{thm:Main 2}, we need following estimates,
which were introduced in \cite{hu2015nearly}. We will show that the
following lemmas holding for the residual stress case. First, we offer
an energy for $\widetilde{u}$ in $D_{h}\backslash\overline{D_{\frac{h}{2}}}$
via a variational method. 
\begin{lem}
\cite{hu2015nearly} Let $\widetilde{u}$ and $u_{0}$ be solutions
of (\ref{eq:Transfored Equation}) and (\ref{eq:Homogeneous Equation}),
respectively. Then there exists a constant $C>0$ depending only on
$\Omega$ such that 
\[
\beta\kappa^{2}\|\widetilde{u}\|_{L^{2}(D_{h}\backslash\overline{D_{\frac{h}{2}}})}^{2}\leq C\|\phi\|_{H^{-1/2}(\partial\Omega)^{N}}\|\widetilde{u}-u_{0}\|_{H^{1/2}(\partial\Omega)^{N}}.
\]
\end{lem}
\begin{proof}
Note that the residual stress $T$ satisfies (\ref{eq:Divergence free})
and (\ref{eq:zero traction boundary}), then $\widetilde{T}$ also
satisfies (\ref{eq:Divergence free}) and (\ref{eq:zero traction boundary})
after transformation by bi-Lipschitz maps. Multiply $\overline{\widetilde{u}}$
(the complex conjugate of $\widetilde{u}$) to (\ref{eq:Transfored Equation})
and integrate by parts, then 
\[
-\int_{\Omega}(\widetilde{\mathcal{C}}\nabla\widetilde{u}):(\nabla\overline{\widetilde{u}})dx+\kappa^{2}\int_{\Omega}\widetilde{\rho}|\widetilde{u}|^{2}dx=-\int_{\partial\Omega}\phi\cdot\overline{\widetilde{u}}dS,
\]
where the notation $:$ denotes the Frobenius inner product. The remaining
proof is the same as Lemma 5.1 in \cite{hu2015nearly} due to (\ref{eq:zero traction boundary}),
so we omit here.
\end{proof}
We define the notations 
\[
\mathcal{N}_{\mathcal{C}^{(0)}}^{\pm}\widetilde{u}(x)=\lim_{\eta\to0}\mathcal{N}_{\mathcal{C}^{(0)}}\widetilde{u}(x\pm\eta\nu)\mbox{ for }x\in\partial D_{h},
\]
where $\nu$ is a unit outer normal of $\partial D_{h}$. For simplicity,
we denote $\Phi^{\pm}(x):=\mathcal{N}_{\mathcal{C}^{(0)}}^{\pm}\widetilde{u}(hx)$
for $x\in\partial D$. 
\begin{lem}
\label{lem:Boundary estimate on cloaked layer}\cite{hu2015nearly}
Let $\widetilde{u}$ and $u_{0}$ be solutions of (\ref{eq:Transfored Equation})
and (\ref{eq:Homogeneous Equation}), respectively. Then there exist
constant $C>0$ depending on $D,\Omega$ but independent of $h,\phi$
such that 
\begin{eqnarray}
\|\Phi^{+}\|_{H^{-3/2}(\partial D)^{N}}^{2} & \leq & C\dfrac{(\gamma+\sqrt{\alpha^{2}+\beta^{2}}h^{-\delta}\omega^{2})^{2}}{\beta\gamma^{2}\omega^{2}}h^{-2-N}\|\widetilde{u}-u_{0}\|_{H^{1/2}(\partial\Omega)^{N}}\nonumber \\
 &  & \times\|\phi\|_{H^{-1/2}(\partial\Omega)^{N}},\label{eq:Phi +}\\
\|\Phi^{-}\|_{H^{-3/2}(\partial D)^{N}}^{2} & \leq & C\dfrac{(\gamma+\sqrt{\alpha^{2}+\beta^{2}}h^{-\delta}\omega^{2})^{2}}{\beta\omega^{2}}h^{2(1+\delta)-N}\|\widetilde{u}-u_{0}\|_{H^{1/2}(\partial\Omega)^{N}}\nonumber \\
 &  & \times\|\phi\|_{H^{-1/2}(\partial\Omega)^{N}}.\label{eq:Phi -}
\end{eqnarray}
\end{lem}
\begin{proof}
Using (\ref{eq:Divergence free}) and (\ref{eq:zero traction boundary})
of the residual stress $T=(t_{jl})_{j,l=1}^{N}$ again, the proof
can be reduced to the isotropic elasticity case, which was proved
in \cite{hu2015nearly}. We refer readers to \cite{hu2015nearly}
for the detailed proof.
\end{proof}
To our best knowledge, there are no layer potential theories for the
elasticity system with residual stress. Therefore, we want to give
appropriate elliptic estimates for the residual stress system by comparing
with the Lamé system as follows. Let $\mathscr{C}=(\mathscr{C}_{ijkl})$
and $\mathscr{C}^{(0)}=(\mathscr{C}_{ijkl}^{(0)})$ with 
\begin{eqnarray*}
\mathscr{C}_{ijkl}(x) & := & \lambda(x)\delta_{ij}\delta_{kl}+\mu(x)(\delta_{ik}\delta_{jl}+\delta_{il}\delta_{jk}),\\
\mathscr{C}_{ijkl}^{(0)} & := & \lambda^{(0)}\delta_{ij}\delta_{kl}+\mu^{(0)}(\delta_{ik}\delta_{jl}+\delta_{il}\delta_{jk}),
\end{eqnarray*}
be isotropic elastic four tensors, where $\lambda(x),\mu(x)$ are
Lamé parameters and $\lambda^{(0)},\mu^{(0)}$ are Lamé constants
which were introduced in (\ref{eq:Residual stress four tensor}) and
(\ref{eq:Lame constants}), respectively. Let $\Delta$ be a $C^{2}$
bounded domain in $\mathbb{R}^{N}$ and $w_{0}$ be a solution of
the following Navier's equation (or homogeneous Lamé system) 
\begin{equation}
\begin{cases}
\mathcal{L}_{0}(w_{0})+\eta^{2}w_{0}=0 & \mbox{ in }\Delta,\\
\mathcal{N}_{\mathcal{\mathscr{C}}^{(0)}}w_{0}=\phi & \mbox{ on }\partial\Delta,
\end{cases}\label{eq:w_0}
\end{equation}
where $\mathcal{L}_{0}(w_{0}):=\nabla\cdot(\mathscr{C}^{(0)}\nabla w_{0})$
and the boundary traction is defined by 
\[
\mathcal{N}_{\mathcal{\mathscr{C}}^{(0)}}w_{0}=\begin{cases}
2\mu^{(0)}\frac{\partial w_{0}}{\partial\nu}+\lambda^{(0)}\nu\nabla\cdot u+\mu\nu^{T}(\partial_{2}w_{0,1}-\partial_{1}w_{0,2}) & \mbox{ for }N=2,\\
2\mu^{(0)}\frac{\partial w_{0}}{\partial\nu}+\lambda^{(0)}\nu\nabla\cdot w_{0}+\mu^{(0)}\nu\times(\nabla\times w_{0}) & \mbox{ for }N=3,
\end{cases}
\]
where $\nu$ is a unit outer normal on $\partial\Delta$, $w_{0}=(w_{0,1},w_{0,2})$
and $\nu^{T}=(-\nu_{2},\nu_{1})\perp\nu$ when $N=2$. Suppose $-\eta^{2}$
is not an eigenvalue of the elliptic operator $\mathcal{L}_{0}$ with
the zero boundary traction condition on $\partial\Delta$, by the
strong convexity (\ref{eq:strong convexity for Lame}), which are
\[
\mu^{(0)}\geq c_{0}>0\mbox{ and }N\lambda^{(0)}+2\mu^{(0)}\geq c_{0}>0,
\]
then we know that (\ref{eq:w_0}) is well-posed. Moreover, the relation
\[
\mathcal{N}_{\mathcal{C}^{(0)}}=\mathcal{N}_{\mathcal{\mathscr{C}}^{(0)}}\mbox{ on }\partial\Delta
\]
is guaranteed by (\ref{eq:zero traction boundary}). We utilize the
following perturbation arguments to construct the near cloaking theory
for the elasticity with residual stress.

Now, let $U_{0}:=u_{0}-w_{0}$ be a reflected solution, where $u_{0}$
and $w_{0}$ are solutions of (\ref{eq:Homogeneous Equation}) and
(\ref{eq:w_0}), respectively. It is easy to see that $U_{0}$ satisfies
the following zero boundary traction problem 
\begin{equation}
\begin{cases}
\mathcal{L}_{0}^{R}U_{0}+\kappa^{2}U_{0}=(\eta^{2}-\kappa^{2})w_{0}-Rw_{0} & \mbox{ in }\Delta,\\
\mathcal{N}_{\mathcal{C}^{(0)}}U_{0}=0 & \mbox{ on }\partial\Delta,
\end{cases}\label{eq:U_0}
\end{equation}
where 
\begin{equation}
Rw_{0}:=\nabla\cdot(T(x)\nabla w_{0})\label{eq:Residual representation}
\end{equation}
and $\mathcal{L}_{0}^{R}$ is the second order elliptic operator appeared
in (\ref{eq:Homogeneous Equation}).

The main estimate of this section is the following one.
\begin{prop}
\label{prop:(Key-estimate)}(Key estimate) Suppose that $\Delta$
is a $C^{2}$ bounded domain in $\mathbb{R}^{N}$ for $N=2,3$. Let
$U_{0}\in H^{1}(\Delta)^{N}$ be solutions of (\ref{eq:U_0}), then
there exist constants $C>0$ independent of $U_{0}$ and $w_{0}$
such that 
\begin{equation}
\|U_{0}\|_{H^{1}(\Delta)^{N}}\leq C\|w_{0}\|_{H^{1}(\Delta)^{N}}.\label{eq:Key Elliptic Estimate}
\end{equation}
 \end{prop}
\begin{proof}
Let $\widehat{U_{0}}$ be a solution of 
\begin{equation}
\begin{cases}
\mathcal{L}_{0}^{R}\widehat{U_{0}}=(\eta^{2}-\kappa^{2})w_{0}-Rw_{0} & \mbox{ in }\Delta,\\
\mathcal{N}_{\mathcal{C}^{(0)}}\widehat{U_{0}}=0 & \mbox{ on }\partial\Delta.
\end{cases}\label{eq:Perturbed equation}
\end{equation}
Note that if $\widehat{U}_{0}$ is a solution of (\ref{eq:Perturbed equation}),
then $\widehat{U_{0}}-\fint_{\Delta}\widehat{U_{0}}$ is also a solution
of (\ref{eq:Perturbed equation}), where $\fint_{\Delta}\widehat{U_{0}}=\dfrac{1}{|\Delta|}\int_{\Delta}\widehat{U_{0}}$.
Without loss of generality, we may assume $\int_{\Delta}\widehat{U_{0}}dx=0$.
Hence, the Poincaré's inequality will be valid for $\widehat{U_{0}}$,
which means there exists a constant $C>0$ independent of $\widehat{U_{0}}$
such that 
\begin{equation}
\|\widehat{U_{0}}\|_{L^{2}(\Delta)^{N}}\leq C\|\nabla\widehat{U_{0}}\|_{L^{2}(\Delta)^{N}}.\label{eq:Poincare for U hat}
\end{equation}
Multiply $\overline{\widehat{U_{0}}}$ (the complex conjugate of $\widehat{U_{0}}$)
on both sides of (\ref{eq:Perturbed equation}) and (\ref{eq:Strong convexity}),
then we can get 
\begin{equation}
\|\nabla\widehat{U_{0}}\|_{L^{2}(\Delta)^{N}}\leq\left|\int_{\Delta}\left\{ (Rw_{0})\cdot\overline{\widehat{U_{0}}}+(\lambda^{2}-\kappa^{2})w_{0}\cdot\overline{\widehat{U_{0}}}\right\} dx\right|.\label{eq:Energy 1}
\end{equation}
Use the integration by parts, (\ref{eq:zero traction boundary}) and
$T(x)\in W^{2,\infty}(\Omega)$, then the Young's inequality yields
that for any $\epsilon>0$, 
\begin{eqnarray}
\left|\int_{\Delta}(Rw_{0})\cdot\overline{\widehat{U_{0}}}dx\right| & \leq & \epsilon\|\nabla\widehat{U_{0}}\|_{L^{2}(\Delta)^{N}}+C(\epsilon)\|\nabla w_{0}\|_{L^{2}(\Delta)^{N}},\label{eq:Young 1}\\
\left|\int_{\Delta}(\lambda^{2}-\kappa^{2})w_{0}\cdot\overline{\widehat{U_{0}}}dx\right| & \leq & \epsilon\|\widehat{U_{0}}\|_{L^{2}(\Delta)^{N}}+C(\epsilon)\|w_{0}\|_{L^{2}(\Delta)^{N}}.\label{eq:Young 2}
\end{eqnarray}
Via (\ref{eq:Energy 1}), (\ref{eq:Young 1}) and (\ref{eq:Young 2}),
we obtain 
\begin{equation}
\|\nabla\widehat{U_{0}}\|_{L^{2}(\Delta)^{N}}\leq C\|w_{0}\|_{H^{1}(\Delta)^{N}},\label{eq:Gradient estimate for U hat}
\end{equation}
where $C>0$ is independent of $\widehat{U_{0}}$ and $w_{0}$. Combine
(\ref{eq:Gradient estimate for U hat}) and (\ref{eq:Poincare for U hat}),
we have 
\begin{equation}
\|\widehat{U_{0}}\|_{H^{1}(\Delta)^{N}}\leq C\|w_{0}\|_{H^{1}(\Delta)^{N}},\label{eq:H^1 for U hat}
\end{equation}
where $C>0$ is independent of $\widehat{U_{0}}$ and $w_{0}$. 

By setting $\mathcal{U}_{0}:=U_{0}-\widehat{U_{0}}$, we have $U_{0}=\widehat{U_{0}}+\mathcal{U}_{0}$
and 
\begin{equation}
\|U_{0}\|_{H^{1}(\Delta)^{N}}\leq C(\|\widehat{U_{0}}\|_{H^{1}(\Delta)^{N}}+\|\mathcal{U}{}_{0}\|_{H^{1}(\Delta)^{N}}).\label{eq:triangular inequality}
\end{equation}
Besides, $\mathcal{U}_{0}$ satisfies 
\begin{equation}
\begin{cases}
\mathcal{L}_{0}^{R}\nabla\mathcal{U}_{0}+\kappa^{2}\mathcal{U}=-\kappa^{2}\widehat{U_{0}} & \mbox{ in }\Delta,\\
\mathcal{N}_{\mathcal{C}^{(0)}}\mathcal{U}_{0}=0 & \mbox{ on }\partial\Delta.
\end{cases}\label{eq:mathcal U equation}
\end{equation}

Recall that the variational formula is 
\[
\mathcal{B}_{\mathcal{C}^{(0)}}(u,v):=\int_{\Delta}\left\{ \sum_{i.j.k.l=1}^{3}C_{ijkl}^{(0)}\dfrac{\partial u_{k}}{\partial x_{l}}\dfrac{\partial\overline{v_{i}}}{\partial x_{j}}-\kappa^{2}\rho u_{i}\overline{v_{i}}\right\} dx,
\]
then we have 
\begin{equation}
\mathcal{B}_{\mathcal{C}^{(0)}}(u,u)\geq c_{0}\|\nabla u\|_{L^{2}(\Delta)^{N}}^{2}-\kappa^{2}\|u\|_{L^{2}(\Delta)^{N}}^{2}\mbox{ for all }u\in H^{1}(\Delta)^{N}.\label{eq:Garding's inequality for perturbed equation}
\end{equation}
Moreover,  by using (\ref{eq:Garding's inequality for perturbed equation}),
one can show the well-posedness of (\ref{eq:mathcal U equation})
provided that $-\kappa^{2}$ is not an eigenvalue of $\mathcal{L}_{0}^{R}$
with vanishing boundary traction. We refer readers to Chapter 4 in
\cite{mclean2000strongly} and Chapter 6 in \cite{marsden1994mathematical}
for more details. When $-\kappa^{2}$ is not an eigenvalue of $\mathcal{L}_{0}^{R}$
with zero boundary traction, the well-posedness of (\ref{eq:mathcal U equation})
in the $L^{2}(\Omega)$ Sobolev space implies that 
\begin{equation}
\|\mathcal{U}_{0}\|_{H^{1}(\Delta)^{N}}\leq C\|\widehat{U_{0}}\|_{L^{2}(\Delta)^{N}},\label{eq:Well-posedness}
\end{equation}
where $C>0$ is a constant independent of $\mathcal{U}_{0}$ and $\widehat{U_{0}}$.
Finally, plug (\ref{eq:H^1 for U hat}) and (\ref{eq:Well-posedness})
into (\ref{eq:triangular inequality}), we can get 
\[
\|U_{0}\|_{H^{1}(\Delta)^{N}}\leq C\|w_{0}\|_{H^{1}(\Delta)^{N}}
\]
as desired. This completes the proof.
\end{proof}
The following trace theorem and inverse trace theorem will be used
in our proof of Theorem \ref{thm:Main 2}. 
\begin{lem}
\label{lem:(Trace-and-inverse-trace)}(Trace and inverse trace theorem)
Let $\varDelta$ be a $C^{2}$ bounded domain in $\mathbb{R}^{N}$,
for any $V\in H^{1}(\varDelta)^{N}$, there exist $c_{1},c_{2}>0$
independent of $V$ such that 
\begin{equation}
c_{1}\|V\|_{H^{1/2}(\partial\varDelta)^{N}}\leq\|V\|_{H^{1}(\varDelta)^{N}}\leq c_{2}\|V\|_{H^{1/2}(\partial\varDelta)^{N}}.\label{eq:Trace theorem}
\end{equation}
\end{lem}
\begin{proof}
This lemma is the standard trace theorem and readers can find the
proof in many literature. For example, see \cite{hsiao2008boundary}
for the detailed proof.
\end{proof}
Now, we prove the following two lemmas by comparing the elasticity
system with residual stress with Lamé system in the domain $\Delta:=\Omega\backslash\overline{D_{h}}$
for $h>0$ small. First, we give a boundary estimate outside small
cavities.
\begin{lem}
Suppose $-\kappa^{2}$ is not an eigenvalue of the elliptic operator
$\mathcal{L}_{0}^{R}$ with zero boundary traction condition in $\Omega$.
Let $v_{0}$ be a solution of 
\begin{equation}
\begin{cases}
\mathcal{L}_{0}^{R}v_{0}+\kappa^{2}v_{0}=0 & \mbox{ in }\Omega\backslash\overline{D_{h}},\\
\mathcal{N}_{\mathcal{C}^{(0)}}v_{0}=\mathcal{N}_{\mathcal{C}^{(0)}}u_{0} & \mbox{ on }\partial D_{h},\\
\mathcal{N}_{\mathcal{C}^{(0)}}v=0 & \mbox{ on }\partial\Omega,
\end{cases}\label{eq:v_0}
\end{equation}
where $u_{0}$ is the solution of (\ref{eq:Homogeneous Equation}).
Then there exists $h_{0}>0$ such that for any $h\in(0,h_{0})$, 
\begin{equation}
\|v_{0}\|_{H^{1/2}(\partial\Omega)^{N}}\leq Ch\|\phi\|_{H^{-1/2}(\partial\Omega)^{N}},\label{eq:Perturbed estimate 1}
\end{equation}
where $C>0$ independent of $h$ and $\phi$.\end{lem}
\begin{proof}
Let $\widehat{v_{0}}$ be a solution of the Lamé system 
\begin{equation}
\begin{cases}
\mathcal{L}_{0}\widehat{v_{0}}+\eta^{2}\widehat{v_{0}}=0 & \mbox{ in }\Omega\backslash\overline{D_{h}},\\
\mathcal{N}_{\mathcal{C}^{(0)}}\widehat{v_{0}}=\mathcal{N}_{\mathcal{C}^{(0)}}u_{0} & \mbox{ on }\partial D_{h},\\
\mathcal{N}_{\mathcal{C}^{(0)}}\widehat{v_{0}}=0 & \mbox{ on }\partial\Omega,
\end{cases}\label{eq:v_0 hat}
\end{equation}
where $-\eta^{2}$ is not an eigenvalue of $\mathcal{L}_{0}$ with
the zero boundary traction on $\Omega$. By the layer potential techniques
for the Lamé system (for example, see \cite{ammari2007polarization,ammari2007asymptotic}),
we have the following estimates 
\begin{eqnarray}
\|\widehat{v_{0}}\|_{H^{1/2}(\partial D_{h})^{N}} & \leq & Ch\|\phi\|_{H^{-1/2}(\partial\Omega)^{N}},\label{eq:Estimate of v_0 on inner boundary}\\
\|\widehat{v_{0}}\|_{H^{1/2}(\partial\Omega)^{N}} & \leq & Ch^{N}\|\phi\|_{H^{-1/2}(\partial\Omega)^{N}},\label{eq:Estimate of v_0 on outer boundary}
\end{eqnarray}
where (\ref{eq:Estimate of v_0 on outer boundary}) was proved in
\cite{hu2015nearly}. For (\ref{eq:Estimate of v_0 on inner boundary}),
we will offer  the proof in our  Appendix. Let $\widehat{v}:=v_{0}-\widehat{v_{0}}$,
where $v_{0}$ and $\widehat{v_{0}}$ are solutions of (\ref{eq:v_0})
and (\ref{eq:v_0 hat}), respectively, then $\widehat{v}$ satisfies
the following boundary value problem 
\[
\begin{cases}
\mathcal{L}_{0}^{R}\widehat{v}+\kappa^{2}\widehat{v}=(\eta^{2}-\kappa^{2})\widehat{v_{0}}-R\widehat{v_{0}} & \mbox{ in }\Omega\backslash\overline{D_{h}},\\
\mathcal{N}_{\mathcal{C}^{(0)}}\widehat{v}=0 & \mbox{ on }\partial D_{h},\\
\mathcal{N}_{\mathcal{C}^{(0)}}\widehat{v}=0 & \mbox{ on }\partial\Omega,
\end{cases}
\]
where $R$ was defined by (\ref{eq:Residual representation}). Applying
the key estimate (\ref{eq:Key Elliptic Estimate}) on the domain $\Delta=\Omega\backslash\overline{D_{h}}$,
there exists $C>0$ independent of $\widehat{v}$ and $\widehat{v_{0}}$
such that 
\[
\|\widehat{v}\|_{H^{1}(\Omega\backslash\overline{D_{h}})^{N}}\leq C\|\widehat{v_{0}}\|_{H^{1}(\Omega\backslash\overline{D_{h}})^{N}}.
\]
By the trace inequality (\ref{eq:Trace theorem}) and (\ref{eq:Estimate of v_0 on inner boundary})
on $\Omega\backslash\overline{D_{h}}$, then it deduces that 
\begin{eqnarray*}
\|v_{0}\|_{H^{1/2}(\partial(\Omega\backslash\overline{D_{h}}))^{N}} & \leq & C\|v_{0}\|_{H^{1}(\Omega\backslash\overline{D_{h}})^{N}}\\
 & \leq & C\left(\|\widehat{v}\|_{H^{1}(\Omega\backslash\overline{D_{h}})^{N}}+\|\widehat{v_{0}}\|_{H^{1}(\Omega\backslash\overline{D_{h}})^{N}}\right)\\
 & \leq & C\|\widehat{v_{0}}\|_{H^{1}(\Omega\backslash\overline{D_{h}})^{N}}\\
 & \leq & C\|\widehat{v_{0}}\|_{H^{1/2}(\partial(\Omega\backslash\overline{D_{h}}))^{N}}\\
 & \leq & Ch\|\phi\|_{H^{-1/2}(\partial\Omega)^{N}}.
\end{eqnarray*}
This completes the proof.\end{proof}
\begin{lem}
Suppose $-\kappa^{2}$ is not an eigenvalue of $\mathcal{L}_{0}^{R}$
with the zero boundary traction in $\Omega$. Consider the following
boundary value problem 
\[
\begin{cases}
\mathcal{L}_{0}^{R}v+\kappa^{2}v=0 & \mbox{ in }\Omega\backslash\overline{D_{h}},\\
\mathcal{N}_{\mathcal{C}^{(0)}}v=\psi & \mbox{ on }\partial D_{h},\\
\mathcal{N}_{\mathcal{C}^{(0)}}v=\phi & \mbox{ on }\partial\Omega.
\end{cases}
\]
Then there exists $h_{0}>0$ such that for any $h\in(0,h_{0})$, 
\begin{equation}
\|v-u_{0}\|_{H^{1/2}(\partial\Omega)^{N}}\leq C\left(h\|\phi\|_{H^{-1/2}(\partial\Omega)^{N}}+h^{N-1}\|\psi(h\cdot)\|_{H^{-3/2}(\partial D)^{N}}\right),\label{eq:Final perturbed estimate}
\end{equation}
where $u_{0}$ is a solution of (\ref{eq:Homogeneous Equation}) and
$C>0$ is independent of $h$, $\varphi$ and $\phi$.\end{lem}
\begin{proof}
Let $\mathbb{V}=u_{0}-v$ in $\Omega\backslash\overline{D_{h}}$,
then $\mathbb{V}$ is a solution of 
\[
\begin{cases}
\mathcal{L}_{0}^{R}\mathbb{V}+\kappa^{2}\mathbb{V}=0 & \mbox{ in }\Omega\backslash\overline{D_{h}},\\
\mathcal{N}_{\mathcal{C}^{(0)}}\mathbb{V}=\mathcal{N}_{\mathcal{C}^{(0)}}u_{0}-\psi & \mbox{ on }\partial D_{h},\\
\mathcal{N}_{\mathcal{C}^{(0)}}\mathbb{V}=0 & \mbox{ on }\partial\Omega.
\end{cases}
\]
Decompose $\mathbb{V}:=\mathbb{V}_{1}-\mathbb{V}_{2}$ such that $\mathbb{V}_{1}$
is a solution of 
\[
\begin{cases}
\mathcal{L}_{0}^{R}\mathbb{V}_{1}+\kappa^{2}\mathbb{V}_{1}=0 & \mbox{ in }\Omega\backslash\overline{D_{h}},\\
\mathcal{N}_{\mathcal{C}^{(0)}}\mathbb{V}_{1}=\mathcal{N}_{\mathcal{C}^{(0)}}u_{0} & \mbox{ on }\partial D_{h},\\
\mathcal{N}_{\mathcal{C}^{(0)}}\mathbb{V}_{1}=0 & \mbox{ on }\partial\Omega,
\end{cases}
\]
and $\mathbb{V}_{2}$ is a solution of 
\[
\begin{cases}
\mathcal{L}_{0}^{R}\mathbb{V}_{2}+\kappa^{2}\mathbb{V}_{2}=0 & \mbox{ in }\Omega\backslash\overline{D_{h}},\\
\mathcal{N}_{\mathcal{C}^{(0)}}\mathbb{V}_{2}=\psi & \mbox{ on }\partial D_{h},\\
\mathcal{N}_{\mathcal{C}^{(0)}}\mathbb{V}_{2}=0 & \mbox{ on }\partial\Omega.
\end{cases}
\]
Making use of (\ref{eq:Estimate of v_0 on outer boundary}), it is
easy to see 
\begin{equation}
\|\mathbb{V}_{1}\|_{H^{1/2}(\partial\Omega)^{N}}\leq Ch^{N}\|\phi\|_{H^{-1/2}(\partial\Omega)^{N}},\label{eq:V1 estimate}
\end{equation}
where $C>0$ is independent of $\mathbb{V}_{1}$ and $\phi$. It remains
to estimate $\mathbb{V}_{2}$. 

Let $\mathbb{W}$ be a solution of 
\[
\begin{cases}
\mathcal{L}_{0}\mathbb{W}+\eta^{2}\mathbb{W}=0 & \mbox{ in }\Omega\backslash\overline{D_{h}},\\
\mathcal{N}_{\mathcal{\mathscr{C}}^{(0)}}\mathbb{W}=\psi & \mbox{ on }\partial D_{h},\\
\mathcal{N}_{\mathcal{\mathscr{C}}^{(0)}}\mathbb{W}=0 & \mbox{ on }\partial\Omega,
\end{cases}
\]
and in Section 5 of \cite{hu2015nearly}, the authors proved that
\[
\|\mathbb{W}\|_{H^{1/2}(\partial\Omega)^{N}}\leq Ch^{N-1}\|\psi(h\cdot)\|_{H^{-3/2}(\partial\Omega)^{N}},
\]
whenever $-\eta^{2}$ is not an eigenvalue of $\mathcal{L}_{0}$ with
zero boundary traction. Set $\mathbb{V}_{3}=\mathbb{V}_{2}-\mathbb{W}$,
then $\mathbb{V}_{3}$ is a solution of 
\[
\begin{cases}
\mathcal{L}_{0}^{R}\mathbb{V}_{3}+\kappa^{2}\mathbb{V}_{3}=(\eta^{2}-\kappa^{2})\mathbb{W}-R\mathbb{W} & \mbox{ in }\Omega\backslash\overline{D_{h}},\\
\mathcal{N}_{\mathcal{C}^{(0)}}\mathbb{V}_{3}=0 & \mbox{ on }\partial(\Omega\backslash\overline{D_{h}}),
\end{cases}
\]
where we have used the same boundary traction, which means $\mathcal{N}_{\mathcal{C}^{(0)}}=\mathcal{N}_{\mathcal{\mathscr{C}}^{(0)}}$
on $\partial(\Omega\backslash\overline{D_{h}})$. Hence, by (\ref{eq:Key Elliptic Estimate})
and (\ref{eq:Trace theorem}), we can derive 
\begin{eqnarray}
\|\mathbb{V}_{3}\|_{H^{1/2}(\partial(\Omega\backslash\overline{D_{h}}))^{N}} & \leq & C\|\mathbb{V}_{3}\|_{H^{1}(\Omega\backslash\overline{D_{h}})^{N}}\label{eq:V_3 estimate}\\
 & \leq & C\|\mathbb{W}\|_{H^{1}(\Omega\backslash\overline{D_{h}})^{N}}\nonumber \\
 & \leq & C\left(h\|\phi\|_{H^{1/2}(\partial\Omega)^{N}}+h^{N-1}\|\psi(h\cdot)\|_{H^{-3/2}(\partial D)^{N}}\right).\nonumber 
\end{eqnarray}
Finally, combining (\ref{eq:V1 estimate}), and (\ref{eq:V_3 estimate})
yields (\ref{eq:Final perturbed estimate}) to be valid, which means
we complete the proof of this lemma.
\end{proof}
\textbf{}\\
Now, we can prove our main result.\textbf{}\\
\textbf{Proof of Theorem \ref{thm:Main 2}.} By Lemma \ref{lem:Boundary estimate on cloaked layer},
if we set $\varphi=\mathcal{N}_{\mathcal{C}^{(0)}}^{+}\widetilde{u}|_{\partial D_{h}}$,
we have $v=\widetilde{u}$ and $\Phi^{+}=\varphi(h\cdot)$. By (\ref{eq:Final perturbed estimate}),
we obtain 
\begin{equation}
\|\widetilde{u}-u_{0}\|_{H^{1/2}(\partial\Omega)^{N}}\leq C\left(h\|\phi\|_{H^{-1/2}(\partial\Omega)^{N}}+h^{N-1}\|\Phi^{+}\|_{H^{-3/2}(\partial D)^{N}}\right).\label{eq:Last perturbed relation}
\end{equation}
Moreover, combine the estimate (\ref{eq:Phi +}) and (\ref{eq:Last perturbed relation}),
and the Young's inequality will lead the desired estimate (\ref{eq:Near cloaking estimte (Main result)})
to hold. We complete the proof. In summary, in this work, we have
built up the nearly cloaking theory for the elasticity system with
residual stress.

\section{Appendix}

In Appendix, we utilize layer potential methods for the Lamé system
to derive (\ref{eq:Estimate of v_0 on inner boundary}). For $x\neq y\in\mathbb{R}^{N}$,
let 
\[
G_{\eta}(x,y)=\begin{cases}
\dfrac{\exp(\sqrt{-1}\eta|x-y|)}{4\pi|x-y|} & \mbox{ when }N=3,\\
\dfrac{i}{4}H_{0}^{(1)}(\eta|x-y|) & \mbox{ when }N=2,
\end{cases}
\]
where $H_{0}^{(1)}$ is the \textit{Hankel function} of the first
kind of order $0$. The Green's tensor $\Pi(x,y)$ for the Lamé system
can be written as 
\[
\Pi(x,y)=\frac{1}{\mu}G_{k_{s}}(x,y)I_{N}+\frac{1}{\eta^{2}}\mbox{grad}_{x}\mbox{grad}_{x}^{T}[G_{k_{s}}(x,y)-G_{k_{p}}(x,y)],
\]
for $x\neq y\in\mathbb{R}^{N}$, $N=2,3$, where 
\[
k_{p}=\frac{\eta}{\sqrt{\lambda+2\mu}},\mbox{ }k_{s}=\frac{\eta}{\sqrt{\mu}}
\]
are compressional and shear constants and $I_{N}$ is the $N\times N$
identity matrix.

Let $\mathcal{O}$ be a bounded simply connected domain in $\mathbb{R}^{N}$
and $\psi(x)$ be a surface density for $x\in\partial\mathcal{O}$,
then we can define the single and double layer potentials in the following
\begin{eqnarray*}
(\mathscr{S}_{\mathcal{O}}\psi)(x) & = & \int_{\partial\mathcal{O}}\Pi(x,y)\psi(y)dS(y),\mbox{ }x\in\mathbb{R}^{N}\backslash\partial\mathcal{O},\\
(\mathscr{D}{}_{\mathcal{O}}\psi)(x) & = & \int_{\partial\mathcal{O}}\Xi(x,y)\psi(y)dS(y),\mbox{ }x\in\mathbb{R}^{N}\backslash\partial\mathcal{O},
\end{eqnarray*}
where $\Xi(x,y)$ is a matrix-valued function with the $i$-th column
vector is 
\[
[\Xi(x,y)]^{T}e_{i}=\mathcal{N}_{\mathscr{C}^{(0)}}[\Xi(x,y)e_{i}]\mbox{ on }\partial\mathcal{O}.
\]
In addition, we set 
\begin{eqnarray*}
(\mathcal{S}_{\partial\mathcal{O}}\psi)(x) & = & \int_{\partial\mathcal{O}}\Pi(x,y)\psi(y)dS(y),\mbox{ }x\in\partial\mathcal{O},\\
(\mathcal{K}_{\partial\mathcal{O}}\psi)(x) & = & \int_{\partial\mathcal{O}}\Xi(x,y)\psi(y)dS(y),\mbox{ }x\in\partial\mathcal{O}.
\end{eqnarray*}

Now, we can prove (\ref{eq:Estimate of v_0 on inner boundary}). Let
$u_{0}$, $v_{0}$, $\phi$ be the same functions defined in previous
sections and let 
\[
V(x):=\int_{\partial D_{h}}\Pi(x,y)\mathcal{N}_{\mathscr{C}^{(0)}}u_{0}dS(y),\mbox{ }x\in\Omega\backslash D_{h},
\]
the authors \cite{hu2015nearly} proved the following estimates 
\begin{equation}
\|V\|_{C(\partial D_{h})}\leq Ch\|\phi\|_{H^{-1/2}(\partial\Omega)^{N}},\mbox{ }\|V\|_{C(\partial\Omega)}\leq Ch^{N}\|\phi\|_{H^{-1/2}(\partial\Omega)^{N}}\label{eq:Appendix 1}
\end{equation}
and 
\begin{equation}
\|\zeta_{1}\|_{L^{2}(\partial D_{h})^{N}}\leq Ch^{(N+1)/2}\|\phi\|_{H^{-1/2}(\partial\Omega)^{N}},\mbox{ }\|\zeta_{2}\|_{L^{2}(\partial\Omega)^{N}}\leq Ch^{N}\|\phi\|_{H^{-1/2}(\partial\Omega)^{N}},\label{eq:Appendix 2}
\end{equation}
where $\zeta_{1}:=w|_{\partial D_{h}}$, $\zeta_{2}:=w|_{\partial\Omega}$.
By the jump relations for double layer potentials, we get 
\[
\zeta_{1}(x)=2\left[(\mathcal{K}_{\partial D_{h}}\zeta_{1})(x)-(\mathscr{D}{}_{\partial\Omega}\zeta_{2})(x)+V(x)\right]\mbox{ when }x\in\partial D_{h}.
\]
Use the similar arguments in \cite{hu2015nearly}, from (\ref{eq:Appendix 1}),
then we can derive 
\[
\|V\|_{H^{1/2}(\partial D_{h})^{N}}\leq C\|V\|_{C(\partial D_{h})^{N}}\leq Ch\|\phi\|_{H^{-1/.2}(\partial\Omega)}.
\]
Finally, from (\ref{eq:Appendix 2}) and the boundedness of $\mathcal{K}_{\partial D_{h}}:L^{2}(D_{h})^{N}\to H^{1}(D_{h})^{N}$,
we have 
\[
\|\mathcal{K}_{\partial D_{h}}\zeta_{1}\|_{H^{1/2}(\partial D_{h})^{N}}\leq Ch^{(N+1)/2}\|\phi\|_{H^{-1/2}(\partial\Omega)^{N}}.
\]
Finally, from (\ref{eq:Appendix 2}) again, we obtain 
\[
\|\mathscr{D}{}_{\partial\Omega}\zeta_{2}\|_{H^{1/2}(\partial D_{h})^{N}}\leq C\|\mathscr{D}{}_{\partial\Omega}\zeta_{2}\|_{C(\partial D_{h})^{N}}\leq Ch^{N}\|\phi\|_{H^{-1/2}(\partial\Omega)^{N}}
\]
as desired, which complete the proof of (\ref{eq:Estimate of v_0 on inner boundary}).

\bibliographystyle{plain}
\bibliography{ref}

\end{document}